\documentclass{article}
\usepackage{a4wide}
\usepackage{authblk}

\usepackage{graphicx}
\usepackage{mathptmx}      

\usepackage[utf8]{inputenc}
\usepackage[english]{babel}
\usepackage[T1]{fontenc}
\usepackage{amsfonts}
\usepackage{amsmath}
\usepackage{amsthm}
\usepackage{amssymb}
\usepackage{graphicx}
\usepackage{graphics}
\usepackage{a4wide}

\usepackage{mathptmx} 
\usepackage{bm}
\usepackage{epstopdf}
\usepackage{tikz}
\usepackage{enumitem}

\usepackage{nicefrac}    
\usepackage{microtype}   
\usepackage{booktabs}    

\let\originalleft\left
\let\originalright\right
\renewcommand{\left}{\mathopen{}\mathclose\bgroup\originalleft}
\renewcommand{\right}{\aftergroup\egroup\originalright}

\usepackage{pifont}
\usepackage{comment}
\usepackage{todonotes}
\usepackage[normalem]{ulem}
\usepackage{etoolbox}
\usepackage{color}
\newtoggle{FinalVersion}
\toggletrue{FinalVersion}
\iftoggle{FinalVersion}{
 \newcommand{\added}[1]{{#1}}
 \newcommand{\deleted}[1]{{}}
}{
 \newcommand{\added}[1]{{\color{blue}#1}}
 \newcommand{\deleted}[1]{{\sout{\color{red}#1}}} 
}

\usepackage[section]{algorithm}
\usepackage{algorithmicx}
\usepackage{algpseudocode}

\algdef{SE}[DOWHILE]{Do}{doWhile}{\algorithmicdo}[1]{\algorithmicwhile\ #1}%

\newtheorem{thm}{Theorem}[section]

\newtheorem{lemma}[thm]{Lemma}

\theoremstyle{definition}
\newtheorem{definition}[thm]{Definition}
\newtheorem{remark}[thm]{Remark}

\newcommand{\norm}[1]{\|#1\|}
\newcommand{\nrm}[1]{|#1|}

\newcommand{\EE}{{\mathbb E}}
\newcommand{\PP}{{\mathbb P}}
\newcommand{\R}{{\mathbb R}}
\newcommand{\bpm}{\begin{pmatrix}}
\newcommand{\epm}{\end{pmatrix}}

\newcommand{\mnmz}{\operatorname*{minimize}}
\newcommand{\mxmz}{\operatorname*{maximize}}
\newcommand{\st}{\operatorname{subject\ to}}
\newcommand{\eps}{{\varepsilon}}

\newcommand{\iLow}{i_{\rm low}}
\newcommand{\iHigh}{i_{\rm high}}
\newcommand{\nmin}{n_{\rm minibatch}}
\newcommand{\nepoch}{n_{\rm epoch}}
\newcommand{\nupdate}{n_{\rm update}}
\newcommand{\smin}{{\rm size}_{\rm minibatch}}
\newcommand{\umax}{u_{\rm max}}
\newcommand{\xdes}{x_{\rm des}}
\newcommand{\Umin}{U_{\rm min}}
\newcommand{\Umax}{U_{\rm max}}
\newcommand{\dmin}{d_{\rm min}}
\newcommand{\dmax}{d_{\rm max}}

\usepackage{pgfplots}
\usepackage{pgfplotstable}
\usetikzlibrary{pgfplots.groupplots}
\usetikzlibrary{pgfplots.statistics}
\pgfplotsset{compat=1.3}
\pgfplotsset{
    legendStyleA/.style={%
        column sep = 0pt,
        legend columns = 1,
        font=\fontsize{6}{5}\selectfont,
        legend cell align={left},
        legend to name = grouplegendA},
    legendStyleB/.style={%
        column sep = 0pt,
        legend columns = 1,
        font=\fontsize{6}{5}\selectfont,
        legend cell align={left},
        legend to name = grouplegendB},        
}

\newtoggle{Draft}
\togglefalse{Draft}
\iftoggle{Draft}{
 \excludecomment{figure}
 
}{
 \pgfplotstableread[col sep = comma]{App1_Control1.csv}\tabAA
 \pgfplotstableread[col sep = comma]{App1_Control2.csv}\tabAB
 \pgfplotstableread[col sep = comma]{App1_Control3.csv}\tabAC
 \pgfplotstableread[col sep = comma]{App1_Control4.csv}\tabAD
 \pgfplotstableread[col sep = comma]{App1_Traj.csv}\tabAE
 \pgfplotstableread[col sep = comma]{App1_Diff.csv}\tabAF
 \pgfplotstableread[col sep = comma]{App2_Var.csv}\tabB
 \pgfplotstableread[col sep = comma]{App2_Traj1.csv}\tabDA
 \pgfplotstableread[col sep = comma]{App2_Traj2.csv}\tabDB
}

\makeatletter
\newenvironment{algorithmaux}[1][htb]{%
    \renewcommand{\ALG@name}{Algorithm (Auxiliary)}
   \begin{algorithm}[#1]%
  }{\end{algorithm}}
\makeatother

\begin{document}

\title{Machine learning approach to chance-constrained problems: An algorithm based on the stochastic gradient descent}
\date{December 11, 2018}

\author[1,2]{Luk\'a\v{s} Adam\thanks{adam@utia.cas.cz}}
\author[2,3]{Martin Branda}

\affil[1]{Southern University of Science and Technology, Shenzhen 518055, China}
\affil[2]{\'UTIA, The Czech Academy of Sciences, Pod Vod\'arenskou v{\v ez}\'i 4, 18208, Prague, Czech Republic}
\affil[3]{Faculty of Mathematics and Physics, Charles University, Sokolovsk\'a 83, 18675, Prague, Czech Republic}
\renewcommand\Authands{ and }

\maketitle

\begin{abstract}
We consider chance-constrained problems with discrete random distribution. We aim for problems with a large number of scenarios. We propose a novel method based on the stochastic gradient descent method which performs updates of the decision variable based only on looking at a few scenarios. We modify it to handle the non-separable objective. A complexity analysis and a comparison with the standard (batch) gradient descent method is provided. We give three examples with non-convex data and show that our method provides a good solution fast even when the number of scenarios is large.
\end{abstract}



\section{Intruduction}

In real-world problems, the data is often stochastic (random). Some examples include uncertain parameters, imprecise measurements or unknown future prices \cite{wiecek.2016}. Since deterministic models do not reflect this fact, they may provide a subpar solution. Stochastic models may provide an alternative. However, changing some variables from deterministic to stochastic brings several issues. The first one is the increased complexity. The second one is the question of how to handle the stochastic variables. As it often provides too conservative results, we do not use the robust optimization approach \cite{bental.nemirovski.1998} and focus instead on the stochastic optimization approach \cite{birge.louveaux.2011}. There, the usual strategy is to replace the stochastic objective by its expectation and impose probability on fulfilling the random constraints. This leads to chance constraints. The expectation describes the average behaviour while the chance constraints specify that the stochastic constraints have to be satisfied with a large probability.

First, we review recent contributions to nonlinear chance-constrained problems.
\cite{Ackooij.Henrion.2014} derived a formula representing the gradients of nonlinear chance constraints in the Gaussian and Student case as a certain integral over the sphere. 
\cite{Geletuetal2017} proposed an algorithm based on solving inner and outer approximations of the chance constrained problems. The approximations consist of two parametric nonlinear programming problems. Asymptotic convergence to an optimal solution is shown. 
\cite{Gonzalez.Heitsch.Henrion.2017} used an approach called spheric-radial decomposition of multivariate Gaussian distributions to solve a demanding problem of gas network design with uncertain demand. All these methods make use of a continuous distribution of the random vector. However, often only an approximation of the true underlying distribution via a finite number of samples is known. In such a case, recent results on nonlinear chance-constrained problems include:
\cite{Fengetal2014} generalized the approach based on the difference of two convex functions by proposing new smooth approximating functions and showed convergence to a stationary point.
\cite{adam.branda.2016a,adam.branda.heitsch.henrion.2018} derived strong and weak necessary optimality conditions. Based on a regularization technique, they proposed an algorithm converging to a stationary point.
\cite{vanAckooij2016} generalized the Benders' decomposition approach for minimizing convex non-differentiable functions over a combinatorial set.
\cite{Xie2017} derived new quantile cuts to strengthen the cutting plane approach for solving mixed-integer nonlinear chance-constrained problems. Since the generation of all cuts is difficult, the authors proposed a practical heuristic approach.
\cite{curtis.2018} introduced a sequential algorithm for solving nonlinear chance constrained problems. The method is based on an exact penalty function which is minimized sequentially by solving quadratic optimization subproblems with linear cardinality constraints.
\added{\cite{KannanLuedtke2018} used a stochastic approximation
method which relies on the multiobjective optimization when the chance constraints are moved to the objective whereas the real objective is bounded as a constraint. The difficult indicator function is replaced by its smooth approximation. They applied the projected stochastic subgradient algorithm and proved a theoretical convergence to stationary points of the smoothed problem which approximate the efficient frontier.}

To the best of our knowledge, all these algorithms either consider a specific continuous distribution or are only suitable for a distribution with a small number of scenarios. Since the number of scenarios should theoretically increase exponentially with the dimension of the random vector, a method which is able to handle a large number of scenarios is needed.

To derive such a method, we seek inspiration in the machine learning method called the stochastic gradient descent \cite{bottou.2018}. It is able to train deep neural nets with billions of samples and millions of decision variables \cite{Krizhevsky.2012}. Its main idea is to use the standard (batch) gradient descent method but instead of computing the gradient on the whole dataset, it computes it only for a small number of samples called the minibatch. The stochastic gradient descent has a direct connection to the coordinate descent method \cite{tseng.2001} where the gradient descent is computed with respect to a few decision variables instead of a few samples. There are several advantages of the stochastic gradient descent over its batch variant:
\begin{enumerate}\itemsep 0pt
\item The gradient computation is much faster and has much lower memory requirements. At the same time, it should not need significantly more iterations to approach the solution because the gradient with respect to any one given sample usually points towards the minimum during the early iterations.
\item Datasets often contain duplicate information or highly correlated samples. Since the batch gradient descent computes the gradients with respect to all samples, unnecessary computation is performed.
\item Due to the noisy gradients, stochastic gradient descent has a higher chance to escape local minima and stationary points for non-convex problems. This also makes it easier to handle nonsmooth functions.
\item Due to the noisy gradients, stochastic gradient descent does not overfit to the training data and generalizes better to unseen samples \cite{poggio.2017}.
\end{enumerate}
All these advantages are closely related to chance-constrained problems. The first two points are general and refer to a lower computational effort. Chance-constrained problems are non-convex even for linear data and when a joint chance constraint is converted into an individual one via the max operator, the constraint is also nonsmooth. Thus, the third point applies. The last point is relevant because for a large dimension of the random vector, sampling provides only a rather crude approximation of the true distribution and overfitting to the training data may decrease the solution quality.

\added{
The biggest disadvantage of the stochastic gradient descent is the possible lack of convergence. Even though there are multiple convergence results \cite{bollapragada2018adaptive,bottou.2018,Ghadimi2016}, they usually show that either a solution is only approached or they require that the gradient is computed from a progressively increasing number of samples. However, the stochastic gradient descent seems to work well in practice \cite{krizhevsky2012imagenet}.
}

Our paper is organized as follows. In Section \ref{sec:preliminaries}, we provide a brief introduction into the chance-constrained problems and into the stochastic gradient descent method. The stochastic gradient method requires that the objective is separable with respect to samples. However, the chance constraint combines all samples together. Thus, in Section \ref{sec:method} we provide an algorithm which handles this obstruction and we provide a computational complexity and comparison to the batch gradient descent in Section \ref{sec:complexity}. Finally, in Section \ref{sec:app1} we provide a brief description of three testing problems and show a good performance of our method in Section \ref{sec:app2}. We stress that we aimed at non-convex problems with a rather large number of scenarios.

\section{Preliminaries}\label{sec:preliminaries}

In this section, we introduce the chance-constrained problems and the stochastic gradient descent.

\subsection{Chance-constrained problems}

Chance-constrained problems are specific types of optimization problems where some constraints have to be satisfied only ``sufficiently'' often and may be violated in some cases. More formally, for a random vector $\xi$, we require that the random constraint $g(x,\xi)\le 0$ is satisfied with probability at least $1-\eps$ for some small $\eps$. With an objective function $f$ and a deterministic constraint set $X$, the chance-constrained problem may be written as
\begin{equation}\label{eq:problem1}
\aligned
\mnmz\ &\EE f(x,\xi)\\
\st\ &\PP(g(x,\xi)\le 0)\ge 1-\eps, \\
&x\in X.
\endaligned
\end{equation}
Here, $\EE$ refers to the expectation and $\PP$ to the probability. The chance constraint states that the $(1-\eps)$-quantile of $g(x,\cdot)$ is at most $0$. Defining the $(1-\eps)$-quantile function formally by
\begin{equation}\label{eq:defin_q}
q(x) := \inf \left \{t|\ \PP(g(x,\xi)\le t)\ge 1-\eps\right\},
\end{equation}
then the chance constraint is equivalent to $q(x)\le 0$. Then problem \eqref{eq:problem1} amounts to
\begin{equation}\label{eq:problem2}
\aligned
\mnmz\ &\EE f(x,\xi)\\
\st\ &q(x)\le 0, \\
&x\in X.
\endaligned
\end{equation}
Lemma \ref{lemma:q} in Appendix \ref{app:q} shows that $q$ is a Lipschitz continuous function under mild conditions. This allows us to solve \eqref{eq:problem2} by the (stochastic) (sub)gradient descent. To this aim, the constraints have to be in a simple form so that we can compute the projection fast. Thus, we penalize the constraint on the quantile to obtain
\begin{equation}\label{eq:problem3}
\aligned
\mnmz\ &\EE f(x,\xi) + \lambda \phi(q(x)) \\
\st\ &x\in X,
\endaligned
\end{equation}
where $\lambda>0$ is the penalization parameter. \added{If the penalization function $\phi$ is non-exact, for example $\phi(z)=\frac12\max\{z,0\}^2$, then the solution \eqref{eq:problem3} provides only an approximation of \eqref{eq:problem2}. However, for increasing $\lambda$ solutions of \eqref{eq:problem3} converge to a solution of \eqref{eq:problem2} under mild conditions \cite{bazaraa2013nonlinear}. Note that we will solve \eqref{eq:problem3} by the projected stochastic gradient descent described in the next section.}

\subsection{Stochastic gradient descent}\label{sec:sgd}

In machine learning, the typical optimization problem takes form
\begin{equation}\label{eq:problem_sgd}
\mnmz\ \frac1S\sum_{i=1}^S h(x,\xi_i).
\end{equation}
Here, $x$ is a decision variable, $h$ is a loss function (usually a discrepancy between predictions and labels) and $\xi_i$ are individual samples. The same problem appears in stochastic optimization, where $\xi_i$ are scenarios (realizations of a random vector) and the goal is to minimize the expectation of $h$.

The simplest approach to solve \eqref{eq:problem_sgd} to apply the (batch) gradient descent, where at iteration $k$ we consider stepsize $\alpha^k>0$ set
\begin{equation}\label{eq:update_batch}
x^{k+1} := x^k - \alpha^k\frac1S\sum_{i=1}^S\nabla_x h(x,\xi_i).
\end{equation}
Since the number of samples $S$ is often large, computing the (batch) gradient in \eqref{eq:update_batch} is time-consuming. The usual strategy is to replace the batch gradient by a stochastic gradient, where we select a subset $I^k$ of $\{1,\dots,S\}$ and perform the update as
\begin{equation}\label{eq:update_sgd}
x^{k+1} := x^k - \alpha^k\frac{1}{\nrm{I^k}}\sum_{i\in I^k}\nabla_x h(x,\xi_i).
\end{equation}
The name stochastic gradient descent reflects the fact that the gradient update is performed only with respect to a stochastic subset of observations.

\section{Solving chance-constrained problems via stochastic gradient descent}\label{sec:method}

In this section, we propose a novel scalable method for solving the chance-constrained problem
\begin{equation}\label{eq:problem_orig}
\aligned
\mnmz\ &\EE f(x,\xi)\\
\st\ &\PP(g(x,\xi)\le 0)\ge 1-\eps, \\
&x\in X.
\endaligned
\end{equation}
We recall that we minimize the expectation of $f(x,\xi)$ while we prescribe the probability that the random constraint $g(x,\xi)\le 0$ is satisfied. Using the quantile functon $q$ defined in \eqref{eq:defin_q}, we may equivalently rewrite the chance constraint into $q(x)\le 0$. When this constraint is penalized a penalization parameter $\lambda$ and a penalization function $\phi$, we arrive at \eqref{eq:problem3}. If $\xi$ has a finite number of scenarios $\{\xi_1,\dots,\xi_S\}$, then this problem amounts to 
\begin{equation}\label{eq:problem}
\aligned
\mnmz\ &\frac1S \sum_{i=1}^S f(x,\xi_i) + \lambda \phi(q(x)) \\
\st\ &x\in X.
\endaligned
\end{equation}
Note that when we drive $\lambda$ to infinity, the solutions of \eqref{eq:problem} will converge (under a constraint qualification) to a solution of \eqref{eq:problem_orig}. Thus, we concentrate on solving \eqref{eq:problem}.

We intend to apply the stochastic gradient descent described in Section \ref{sec:sgd}. To this aim, the projection onto $X$ has to be simple and the objective function has to be separable as in \eqref{eq:problem_sgd}. This is true for the first part of our objective \eqref{eq:problem}. However, the quantile function $q$ in the second part of \eqref{eq:problem} combines all scenarios and thus, it is not separable. In this part, we will extend the stochastic gradient descent to handle this obstruction.

\subsection{How to compute derivatives?}\label{sec:derivative}

First, we will compute the derivative for $q$. Since there are finite number of scenarios, the quantile is realized at some scenario, see Lemma \ref{lemma:q} in Appendix \ref{app:q}. Formally, for every $x$, there exists some index $i(x)$ such that
\begin{equation}\label{eq:q_equality_1}
q(x) = g(x,\xi_{i(x)}).
\end{equation}
If this index is unique and $g(\cdot,\xi_{i(x)})$ is differentiable at $x$, then $q$ is differentiable at $x$ and its derivative equals to
$$
\nabla q(x) = \nabla_x g(x,\xi_{i(x)}).
$$
Then the derivative of the objective function \eqref{eq:problem} can be computed via the chain rule and equals to
\begin{equation}\label{eq:nabla_F}
\frac1S\sum_{i=1}^S\nabla_x f(x,\xi_i) + \lambda \phi'(q(x))\nabla_x g(x,\xi_{i(x)}).
\end{equation}
Note that the last part of the gradient depends only on one scenario $i(x)$. Having the gradient at hand, it is simple to write the gradient descent update
\begin{equation}\label{eq:update_batch_ccp}
\aligned
y^{k+1} &:= x^k - \alpha^k\left( \frac1S \sum_{i=1}^S\nabla_x f(x^k,\xi_i) + \lambda \phi'(q(x^k))\nabla_x g(x^k,\xi_{i(x)}) \right), \\
x^{k+1} &:= P_X(y^{k+1}),
\endaligned
\end{equation}
where $\alpha^k>0$ is the stepsize, $P_X$ is the projection onto the feasible set $X$ and $k$ denotes the iteration index. We summarize the (batch) gradient descent in Algorithm \ref{alg:batch}.

\begin{algorithmaux}[!ht]
\begin{algorithmic}[1]
\State \label{alg:batch1}Set index $k\gets 0$
\State \label{alg:batch2}Initialize variable $x^0$
\While{\textbf{not }termination criterion}\label{alg:batch3}
\State \label{alg:batch4}Compute $g_i^k\gets g(x^k,\xi_i)$ for all $i\in \{1,\dots,S\}$ \Comment{Update $g$ }
\State \label{alg:batch5}Find quantile $q^k$ of $\{g_i^k\}_{i=1}^S$ and the index $i(x^k)$ from \eqref{eq:q_equality_1} realizing the quantile \Comment{Find quantile}
\State \label{alg:batch6}Update $x^k$ according to \eqref{eq:update_batch_ccp} \Comment{Update $x$}
\State \label{alg:batch7}Increase $k$ by one
\EndWhile
\end{algorithmic}
\caption{Batch gradient descent for solving \eqref{eq:problem}}
\label{alg:batch}
\end{algorithmaux}

Unfortunately, update \eqref{eq:update_batch_ccp} requires the quantile $q(x^k)$, which in turn needs the evaluation of $g(x,\xi_i)$ for all scenarios $i\in\{1,\dots,S\}$. Since this is a costly update, we will suggest a new update which evaluates $g$ only on a (small) number of samples $I^k$ called the minibatch. The main idea is to use auxiliary variables $z_i^k$ which approximate $g(x^k,\xi_i)$ and to consecutively update this approximation by setting
\begin{equation}\label{eq:update_z}
z_i^k := \begin{cases} g(x^k,\xi_i) &\text{if }i\in I^k, \\ z_{i}^{k-1} &\text{otherwise.} \end{cases}
\end{equation}
Since $k$ is the iteration index, $z_i^k$ contains the evaluation of $g(\cdot,\xi_i)$ for some (possibly delayed) value of $x$.

Then we compute the approximation $q^k$ of the quantile $q(x^k)$ defined in \eqref{eq:defin_q} as a quantile of $\{z_i^k\}_{i=1}^S$, which amounts to solving
\begin{equation}\label{eq:update_q}
q^k := \inf \left\{t\mid\ \frac1S \sum_{i=1}^S \chi(z_i^k \le t) \ge 1-\eps\right\}.
\end{equation}
Here, $\chi$ is the characteristic ($0$-$1$) function checking if $z_i^k \le t$ is satisfied. Similarly to \eqref{eq:q_equality_1}, there is some $i^k$ such that
\begin{equation}\label{eq:update_i}
q^k = z_{i^k}^k.
\end{equation}
Then we can approximate the gradient in \eqref{eq:nabla_F} by
\begin{equation}\label{eq:update_nabla}
\frac{1}{\nrm{I^k}}\sum_{i\in I^k}\nabla_x f(x^k,\xi_i) + \lambda \phi'(q^k)\nabla_x g(x^k,\xi_{i^k}).
\end{equation}
and the next iterate in \eqref{eq:update_batch_ccp} by
\begin{equation}\label{eq:update_x}
\aligned
y^{k+1} &:= x^k - \alpha^k\left( \frac{1}{\nrm{I^k}}\sum_{i\in I^k}\nabla_x f(x^k,\xi_i) + \lambda \phi'(q^k)\nabla_x g(x^k,\xi_{i^k}) \right), \\
x^{k+1} &:= P_X(y^{k+1}).
\endaligned
\end{equation}

There are three differences between \eqref{eq:update_batch_ccp} and \eqref{eq:update_x}. First, the expectation of the gradient of $f$ with respect to all samples is replaced by its expectation with respect to $I^k$. Second, the derivative of the penalization function at the exact quantile $\phi'(q(x^k))$ is replaced by its derivative at the approximative quantile $\phi'(q^k)$. Third, index $i(x^k)$ satisfying \eqref{eq:q_equality_1} is replaced by $i^k$ satisfying \eqref{eq:update_i}. Naturally, this makes the update \eqref{eq:update_x} inexact. However, its main strength lies in the fact that $g$ is evaluated only at the active minibatch $I^k$ in \eqref{eq:update_z} and at one index $i^k$ in \eqref{eq:update_nabla}. Thus, $g$ is computed only at a small number of indices and thus update \eqref{eq:update_x} is much faster than \eqref{eq:update_batch_ccp}, especially if the computation of $g$ is difficult. 

\added{
In the analysis above, we required that $i^k$ is unique and $g(\cdot,\xi_{i^k})$ is differentiable at $x^k$. If the former is not the case, then we replace the gradient in \eqref{eq:update_x} by a subgradient. If the latter is not the case, then we select any $i^k$ satisfying \eqref{eq:update_i} at random. Note that in the numerical experiments, the index $i^k$ was always unique. This conforms with the fact that Lipschitz continuous functions are differentiable almost everywhere due to the Rademacher's theorem.
}

\subsection{Numerical implementation}\label{sec:algorithm}

\added{
We summarize the ideas from the previous section into Algorithm \ref{alg:sgd}. In initialization steps \ref{alg:step_init1}-\ref{alg:step_init4} we choose an initial point $x^1$ and compute $z^0$ with components $z_i^0=g(x^1,\xi_i)$. In step \ref{alg:step_perm1} we find a random permutation of all indices which will be the basis for the minibatch selection. For simplicity we assume that every minibatch contains $\smin$ samples, thus $\{1,\dots,S\}$ can be split into
\begin{equation}\label{eq:defin_nmin}
\nmin := \left\lfloor \frac{S}{\smin}\right\rfloor
\end{equation}
minibatches. In step \ref{alg:step_loop2} we perform a loop over these minibatches. In step \ref{alg:step_update1} we get the minibatch $I^k$ based on the random permutation $\theta$ and we update $g(x^k,\xi_i)$ on this minibatch in step \ref{alg:step_update2}. Step \ref{alg:step_update3} updates $z_i^k$ and the next step finds the quantile $q^k$ and the corresponding index $i^k$. Finally, we update $x^k$ in step \ref{alg:step_update6}, increase $k$ and reiterate. There is a more efficient way of computating $q^k$. However, since it is rather technical and the only benefit is a faster computation, we postpone it to Appendix \ref{app:quantile}.

\begin{algorithmaux}[!ht]
\begin{algorithmic}[1]
\State \label{alg:step_init1}Initialize variable $x^1$
\State \label{alg:step_init2}Set $z_i^0=g(x^1,\xi_i)$ for all $i=1,\dots,S$
\State \label{alg:step_init4}Set index $k\gets 1$
\While{\textbf{not }termination criterion}\label{alg:step_loop1}\Comment{Epoch index}
\State \label{alg:step_perm1}Get a random permutation $\theta$ of $\{1,\dots,S\}$\Comment{Randomly shuffle scenarios}
\For{$j=1,\dots, \nmin$}\label{alg:step_loop2}\Comment{Loop within an epoch}
\State \label{alg:step_update1}$I^k\gets \cup\{\theta(i)|\ i\in[(j-1)\smin+1, j\smin]\}$\Comment{Determine minibatch}
\State \label{alg:step_update2}Compute $g_i^k\gets g(x^k,\xi_i)$ for all $i\in I^k$ \Comment{Update $g$ on minibatch}
\State \label{alg:step_update3}Update $z_i^k$ based on \eqref{eq:update_z}\Comment{Update $z$}
\State \label{alg:step_update4}Find quantile $q^k$ of $\{z_i^k\}_{i=1}^S$ and the index $i^k$ from \eqref{eq:update_i} realizing the quantile\Comment{Find quantile}
\State \label{alg:step_update6}Update $x^k$ according to \eqref{eq:update_x} \Comment{Update $x$}
\State \label{alg:step_update7}Increase $k$ by one
\EndFor
\EndWhile
\end{algorithmic}
\caption{Stochastic gradient descent for solving \eqref{eq:problem}}
\label{alg:sgd}
\end{algorithmaux}
}
The procedure described in Algorithm \ref{alg:sgd} solves problem \eqref{eq:problem} for one fixed $\lambda$. When we are satisfied with the current solution, we increase $\lambda$ and use the terminal value from the previous $\lambda$ as the starting value for the next $\lambda$. This provides a solution to the chance-constrained problem \eqref{eq:problem_orig} and the procedure is summarized in Algorithm \ref{alg:lambda}.

\begin{algorithm}[!ht]
\begin{algorithmic}[1]
\For{$\lambda^1<\lambda^2<\dots<\lambda^L$}
\State Employ Algorithm \ref{alg:sgd} with starting value $x^{l-1}$ and $\lambda=\lambda^l$ to get $x^l$
\EndFor
\end{algorithmic}
\caption{For solving the chance-constrained problem \eqref{eq:problem_orig}}
\label{alg:lambda}
\end{algorithm}

\subsection{Convergence analysis}

I\added{n this section, we discuss the convergence of our algorithm. There are multiple convergence results for the stochastic gradient descent \cite{bollapragada2018adaptive,bottou.2018,Ghadimi2016}; however, we are not aware of any which handles our case of a biased gradient estimate and a non-differentiable non-convex function $q$. Even though we do not provide a formal proof, we give an explanation of why we will observe convergence in the numerical experiments. We start with the following lemma whose proof we postpone to Appendix \ref{app:q}.

\begin{lemma}\label{lemma:q_conv}
Let $X$ be a compact convex set, $f$ and $g(\cdot,\xi_i)$ be Lipschitz continuous functions for all $i$ and let $\phi(z) = \frac12\max\{0,z\}^2$. Moreover, assume that $\alpha^k\to 0$ or that $\{x^k\}$ is convergent. Then $\nrm{q(x^k)-q^k}\to 0$.
\end{lemma}

In the rest of this section, we assume that the assumptions of Lemma \ref{lemma:q_conv} are satisfied. We recall that the true \eqref{eq:nabla_F} and estimated \eqref{eq:update_nabla} gradients equal to
$$
\aligned
h(x^k) &= \frac1S\sum_{i=1}^S\nabla_x f(x^k,\xi_i) + \lambda \phi'(q(x^k))\nabla_x g(x^k,\xi_{i(x^k)}) \\
\hat h(x^k) &= \frac{1}{\nrm{I^k}}\sum_{i\in I^k}\nabla_x f(x^k,\xi_i) + \lambda \phi'(q^k)\nabla_x g(x^k,\xi_{i^k}).
\endaligned
$$
The usual requirement for convergence proofs is that $\EE\hat h(x)$ is close to $h(x)$, either in a direction or in a norm \cite{bottou.2018}. The most common case is the requirement of $\EE\hat h(x)=h(x)$, thus $\hat h(x)$ needs to be an unbiased estimate of $h(x)$. We immediately observe that this holds true for the first part of $\hat h(x)$, namely we have
$$
\EE\left(\frac{1}{\nrm{I^k}}\sum_{i\in I^k}\nabla_x f(x^k,\xi_i)\right) = \frac1S\sum_{i=1}^S\nabla_x f(x,\xi_i),
$$
where the expectation is taken with respect to uniform sampling of the minibatch $I^k$. For the second part we due to Lemma \ref{lemma:q_conv} observe that
\begin{equation}\label{eq:phi_diff}
\nrm{\phi'(q(x^k)) - \phi'(q^k)}=\nrm{\max\{q(x^k),0\} - \max\{q^k,0\}}\le \nrm{q(x^k)-q^k}\to 0.
\end{equation}
It remains to estimate the last term $\norm{\nabla_x g(x^k,\xi_{i(x^k)}) - \nabla_x g(x^k,\xi_{i^k})}$.

Recall that $i(x^k)$ is the index of the quantile of $\{g(x^k,\xi_i)\}_{i=1}^S$ and similarly $i^k$ is the index of the quantile of $\{z_i^k\}_{i=1}^S$. From the proof of Lemma \ref{lemma:q_conv} we observe that $\nrm{g(x^k,\xi_i)-z_i^k}\to 0$ for all $i$. Consider now the case when $i(x^k)$ is unique. If $k$ is sufficiently large (and if the spread of $g(x^k,\xi_i)$ around $g(x^k,\xi_{i(x^k)})$ is sufficiently large), then $i(x^k)=i^k$. This due to \eqref{eq:phi_diff} implies that
$$
\aligned
\nrm{\lambda \phi'(q(x^k))\nabla_x &g(x^k,\xi_{i(x^k)}) - \lambda \phi'(q^k)\nabla_x g(x^k,\xi_{i^k})} \\
&= \lambda\nrm{(\phi'(q(x^k))-\phi'(q^k))\nabla_x g(x^k,\xi_{i(x^k)})} \\
&\le \lambda\nrm{(\phi'(q(x^k))-\phi'(q^k))}\norm{\nabla_x g(x^k,\xi_{i(x^k)})} \to 0.
\endaligned
$$
Thus, the second part of $\hat h(x)$ is close to the second part $h(x)$. If $i(x^k) $ is not unique, then Remark \ref{remark:indices} from Appendix \ref{app:q} suggests that at least in some cases, $q$ is still differetiable at $x^k$ and the gradient equals to a convex hull of $\nabla_x g(x,\xi_i)$ where the convex hull is taken with respect to $i$ which coincide with $i(x^k)$. Then selecting an index $i(x^k)$ at random provides a good approximation of the gradient.

To summarize, both the estimate $\hat h(x^k)$ and the true gradient $h(x^k)$ consist of two parts. The expectation of the difference of the first parts is zero. The difference of the second parts goes to zero whenever the index $i^k$ is unique. Since these are the two core properties used in convergence proofs of the stochastic gradient descent, and since $i^k$ was always unique in the numerical experiments, we believe that this is the explanation why we observed convergence in the numerical section.
}

\section{Complexity analysis and comparison with the batch algorithm}\label{sec:complexity}

In this section, we provide a complexity analysis of Algorithm \ref{alg:sgd} and we show the benefits of our algorithm over the standard batch gradient from Algorithm \ref{alg:batch}. In machine learning, the crucial term is the epoch. We provide the definition now.

\begin{definition}\label{def:epoch}
One epoch is the time when $g$ is evaluated $S$ times.
\end{definition}

\noindent In other words, the epoch is the time during which $g$ looks at the whole dataset and evaluates each scenario for some $x$. Since the evaluation of $g$ is often the most demanding computation, epoch acts as an indicator for the computational time.

In the batch algorithm from Algorithm \ref{alg:batch}, one epoch equals to one \texttt{while} loop in step \ref{alg:batch3}. Denoting $g$ the complexity of computing $g(\cdot,\xi_i)$ for one sample $i$, then the complexity of step \ref{alg:batch4} is $O(gS)$. For step \ref{alg:batch5} there are algorithms \cite{knuth.1997} for finding the quantile with complexity $O(S)$. In total, the complexity is $O(gS)$ and during one epoch we will perform one update of $x$.

\added{
In the stochastic gradient descent from Algorithm \ref{alg:sgd}, the epoch consists of one \texttt{while} loop in step \ref{alg:step_loop1}. The permutation in step \ref{alg:step_perm1} can be found via the Fisher-Yates shuffle \cite{knuth.1997} with complexity $O(S)$. Computing $g(x^k,\xi_i)$ on the minibatch $I^k$ in step \ref{alg:step_update2} has complexity $O(g\smin)$. The quantile in step \ref{alg:step_update4} can be found in $O(S)$. All other steps are negligible. Thus, the complexity of the inner loop in step \ref{alg:step_loop2} amounts to
$$
O(g\smin + S)
$$
Since for one epoch we perform this update $\nmin$ times, due to \eqref{eq:defin_nmin} the complexity of one epoch equals to
$$
O(gS + \nmin S).
$$
During one epoch, $\nmin$ updates of $x$ are performed. 
}

\begin{table}[!ht]
\centering
\caption{Comparison of the complexity of the batch and stochastic gradient descent methods for solving the (penalization of) chance-constrained problem \eqref{eq:problem}.}
\label{tab:complexity}
\begin{tabular}{lll}
\toprule
& Complexity & Updates of $x$ \\
Batch gradient descent & $O(gS)$ & $1$ \\
Stochastic gradient descent & $O(gS + \nmin S)$ & $\nmin$ \\ \midrule
Batch gradient descent & $O\big(g\sqrt{S}S\big)$ & $\sqrt{S}$ \\
Stochastic gradient descent & $O\big((g+\sqrt{S})S\big)$ & $\sqrt{S}$ \\
\bottomrule
\end{tabular}
\end{table}

We summarize the complexity in Table \ref{tab:complexity}. The first two rows follow directly from the discussion above. For a simple comparison, we choose $\smin=\nmin=\sqrt{S}$, for which the complexity for $\sqrt S$ epochs reduces to $O(g\sqrt S S)$ and $O((g + \sqrt S)S)$, respectively. Thus, we see that the stochastic gradient descent provides a clear benefit whenever the computation of $g$ is time-consuming. 

\section{Applications: Description}\label{sec:app1}

In this section, we describe three applications to test the performance of our Algorithm \ref{alg:lambda}. The numerical results are postponed to Section \ref{sec:app2}.

\subsection{Application 1: Optimal control of fish population}\label{sec:app1_1}

This application is adapted from \cite[Chapter 8]{lenhart.2007}, where the authors provided the optimal fishing strategy for a deterministic fishery model without a terminal condition. The model assumes the logistic population evolution
$$
\dot x = rx\left(1-\frac xK\right),
$$
where $x(t)$ is the number of fish, $r$ is the growth rate, $K$ is the carrying capacity and the initial condition $x(0)=x_0$ is satisfied. We introduce control variable $u(t)$ which measures the fishing rate. Then model changes into
$$
\dot x = rx\left(1-\frac xK\right) - ux.
$$
Discretization this ODE via the forward Euler scheme with time step $\Delta t$ leads to
\begin{equation}\label{eq:app1_par1}
x_{t+1} = x_t + \Delta t\left(rx_t - \frac 1Krx_t^2 - u_tx_t\right).
\end{equation}
The profit from fishing may be written as
$$
\Delta t\sum_{t=0}^{T-1}\left( p_tu_tx_t - d_tu_t^2x_t^2 - c_tu_t \right)
$$
where $\Delta t$ is the time step, $p$ is the fish price, $d$ the cost of diminishing returns and $c$ the fishing costs. Putting this all together, we obtain an optimization problem
\begin{equation}\label{eq:app1_problem_det1}
\aligned
\mxmz\ &\Delta t\sum_{t=0}^{T-1}\left( p_tu_tx_t - d_tu_t^2x_t^2 - c_tu_t \right)
 \\
\st\ &\text{system }\eqref{eq:app1_par1}\text{ holds true},\\
&u(t) \in [0,\umax],
\endaligned
\end{equation}
where $\umax$ is the maximal fishing rate.

The solution of this problem will likely result in an empty fishery, which is not optimal for a long-term development. Thus, we add the constraint $x_T\ge \xdes$. We do not need to consider the non-negativity constraint $x_t\ge 0$. Indeed, if $x$ is negative, then the solution may be improved by considering $u=0$. Since the carrying capacity $K$, the growth rate $r$, the initial state $x_0$ and the time-dependent future prices $p,d,c$ are not known precisely, it makes sense to consider them as stochastic variables. Then $x$ becomes a stochastic variable as well and the deterministic constraint $x_T\ge\xdes$ needs to be changed into a stochastic one. This leads to the problem
\begin{equation}\label{eq:app1_problem}
\aligned
\mxmz\ &\Delta t\sum_{t=0}^{T-1}\EE\left( p_tu_tx_t - d_tu_t^2x_t^2 - c_tu_t \right)\\
\st\ &\text{system }\eqref{eq:app1_par1}\text{ holds true},\\
&\PP(x_T \ge \xdes)\ge 1-\eps, \\
&u_t \in [0,\umax].
\endaligned
\end{equation}
For every $u$, we are able to compute $x$ in a unique manner. Thus, the decision variable is only $u$ and $x(u)$ can be considered as the state variable. Then, problem \eqref{eq:app1_problem} fits into setting \eqref{eq:problem_orig} and we may apply Algorithm \ref{alg:lambda} to solve it. The derivates are computed via the backpropagation technique described in Appendix \ref{app:derivatives}. We would like to note that since $x_t$ does not depend on $p_t$, these random variables are independent. Thus, we have $\EE p_tu_tx_t = u_t\EE p_t\EE x_t$ and the prices $p_t$, $d_t$ and $c_t$ may be replaced by their expectations. Here, we keep the original more complex problem \eqref{eq:app1_problem} to simulate the situation where our algorithm is applied in a brute force manner. We comment more on this in Appendix \ref{app:diff}.

\subsection{Application 2: Optimal control of electrostatic separator}\label{sec:app1_2}

Recently, a great emphasis has been put on the circular economy where resources are reused instead of being discarded. One of the major problems is the abundance of plastics, for example, the packaging is often discarded immediately after being used \cite{ragaert.2017}. Since there are multiple types of plastics, the waste stream usually does not contain only one type but their mixture. Since each plastic has a different recycling procedure, a necessary preliminary step before recycling is their separation \cite{doetterl.2016}. 

One of the separation possibilities is the electrostatic free-fall separator \cite{wu.2013} depicted in Figure \ref{fig:separator2}. Two types of particles are charged with opposite polarities, placed into the feeder and then dropped into the separator with an electrostatic field generated by two parallel electrodes. Due to the gravity, the particles fall downwards and due to the electrostatic field and opposite polarities, one of the plastic types falls leftwards while the other one rightwards. Thus, the particles separate.

There were several attempts to optimize the shape of the free-fall separator \cite{mach.2014}. Here, we consider the simplified version with parallel electrodes. We consider three forces acting on the particles: the gravitational force $F_g=mg$, the Coulomb force $F_c = \frac{QU}{d_r-d_l}$ and the air drag $F_a=\frac 12 CS\rho cv^2$. Here, $m$ is the mass of the particle, $Q$ its charge, $U$ the voltage, $d_r-d_l$ the distance of the electrodes, $C$ the drag coefficient, $S$ the particle cross-section, $\rho$ the density of air and $v$ the particle speed. Since the electrodes are assumed to be parallel, the Coulomb force $F_c$ reduces to the simple form above. The equations for particle position $s$ and velocity $v$ in components $(x,y)$ equal to
\begin{equation}\label{eq:app2_par1}
\aligned
\dot s_x &= v_x, \\
\dot s_y &= v_y, \\
\dot v_x &= \frac{QU}{m(d_r-d_l)} - \frac1{2m}CS\rho v_x\sqrt{v_x^2+v_y^2}, \\
\dot v_y &= g - \frac1{2m}CS\rho v_y\sqrt{v_x^2+v_y^2}.
\endaligned
\end{equation}
We add the impact of particles on the electrodes rather informally by writing
\begin{equation}\label{eq:app2_par2}
v_x\text{ changes sign if }s_x=d_l\text{ or }s_x=d_r.
\end{equation}
Due to this constraint, the system is non-differentiable. As the control variables we consider the voltage $U$, the position of the left electrode $d_l$ and the position of the right electrode $d_r$. The random variables include the charge $Q$, the mass $m$ and the initial position $(s_x,s_y)$ and the initial velocity $(v_x,v_y)$.

Since the cost of the electrostatic separator is proportional to the voltage on the electrodes, we want to minimize it. At the same time, we want to achieve a high-quality separation. Since the positively charged particles are supposed to fly right, we expect $s_{x,{\rm pos}}\ge\xdes$, where $\xdes$ is some positive desired state. Similarly, the negatively charged particles are supposed to fly left and thus, we want them to satisfy $s_{x,{\rm neg}}\le-\xdes$. This gives rise to the following problem
\begin{equation}\label{eq:app2_problem}
\aligned
\mnmz\ &U \\
\st\ &\text{system }\eqref{eq:app2_par1}\text{ and }\eqref{eq:app2_par2}\text{ holds true}, \\
&\PP(s_{x,{\rm pos}}(T) \ge \xdes,\ s_{x,{\rm neg}}(T) \le -\xdes)\ge 1-\eps, \\
&U\in [\Umin,\Umax],\ d_l\in[-\dmax,-\dmin],\ d_r\in[\dmin,\dmax] .
\endaligned
\end{equation}
Similarly to the previous application, we can consider the decision variables as only $U$, $d_l$ and $d_r$ while considering $s_x$, $s_y$, $v_x$ and $v_y$ as the dependent state variables. The ODE is again discretized via the forward Euler scheme and the derivatives are computed via the backpropagation technique described in Appendix \ref{app:backprop}.

\subsection{Application 3: Optimal design of gas network}\label{sec:app1_3}

We follow the gas network described in \cite{gotzes.heitsch.henrion.schultz.2016} by an injection node $0$, withdrawal nodes $\{1,\dots,n\}$ and a set of pipes (edges) with their pressure drop coefficients $\Phi_e$. For each node $i$, there is a stochastic demand $\xi_i$ which is assumed to follow a known distribution. The goal in \cite{adam.branda.heitsch.henrion.2018} was to design the network such that the demand is satisfied with a high probability. This was specified by controlling the upper-pressure bounds  $p_i^{\rm max}$ while the lower pressure bounds $p_i^{\rm min}$ were normalized to one.

For tree-structured networks without cycles, the authors in \cite{gotzes.heitsch.henrion.schultz.2016} showed that a random demand $\xi$ can be satisfied if and only if
\begin{equation}\label{eq:app3_g}
\aligned
(p_0^{\rm min})^2 &\leq (p_i^{\rm max})^2 + h_i(\xi),\ i = 1, \dots, n,\\
(p_0^{\rm max})^2 &\geq (p_i^{\rm min})^2 + h_i(\xi),\ i = 1, \dots, n,\\
(p_i^{\rm max})^2 + h_i(\xi) &\geq (p_j^{\rm min})^2 + h_j(\xi),\ i,j = 1, \dots, n.
\endaligned
\end{equation}
Here, functions $h_i(\xi)$ can be computed by
$$
h_i(\xi) = \sum_{e \in \Pi(i)} \Phi_e \left( \sum_{j \succeq \pi(e)} \xi_j \right)^{2},
$$
where $\Pi(i)$ denotes the unique directed path (edges) from the root node $0$ to node $i$, $\pi(e)$ is the end node of edge $e$ and $j \succeq i$ means that the unique path from root to $j$ passes through $i$.

There are many ways of defining the objective. The simplest way is to minimize the total upper pressure bounds, which results in
\begin{equation}\label{eq:app3_problem}
\aligned
\mnmz\ &\sum_{i=1}^n p_i^{\rm max}\\
\st\ &\mathbb P(\text{system \eqref{eq:app3_g} is fulfilled}) \ge 1-\varepsilon, \\
&p_i^{\rm max}\ge 1.
\endaligned
\end{equation}
To apply our algorithm, we need to have only one chance constraint. Since \eqref{eq:app3_g} contains multiple constraints, we employ the standard trick of passing to the maximum. This gives rise to the combined single constraint
$$
\max_{i=1,\dots,n}\left((p_i^{\rm max})^2 - \tilde h_i(\xi) \right) \ge 0,
$$
where $\tilde h_i(\xi)$ combines the values of $h_j(\xi)$ and $p_j^{\rm min}$. Note that this is a non-convex and a nonsmooth constraint. However, as mentioned in the introduction, this should not be a (big) hurdle for our algorithm due to its stochastic nature.

\section{Applications: Numerical results}\label{sec:app2}

In this section, we apply our Algorithm \ref{alg:lambda} to the chance-constrained problems described in Section \ref{sec:app1}. We summarize these applications in Table \ref{tab:data1}. Note that in all cases we considered a rather large number of $S=100{,}000$ scenarios. Moreover, two of these applications contains an ODE in the constraints, one is nonsmooth and one contains joint chance constraints. In all three applications, the function $g(\cdot,\xi)$ is non-convex.

\begin{table}[!ht]
\centering
\caption{Summary of the used problems: number of variables $n$, number of scenarios $S$, dimension of the random vector $\xi$, type of chance constraints, type of algebraic constraints and special features.}
\label{tab:data1}
\begin{tabular}{@{}llllllll@{}}
\toprule
& Label & $n$ & $S$ & $\operatorname{dim}\xi$ & CCP type & Constraints & Speciality \\
Fishing & \eqref{eq:app1_problem} & $1000$ & $100{,}000$ & $3+3n$ & Individual & Box & ODE \\
Separator & \eqref{eq:app2_problem} & $3$ & $100{,}000$ & $4$ & Individual & Box & ODEs, Nonsmooth \\
Gas network & \eqref{eq:app3_problem} & $12$ & $100{,}000$ & $n$ & Joint & Box & - \\
\bottomrule
\end{tabular}
\end{table}

In Table \ref{tab:data2} we summarize the used parameters. If multiple values were used, they are separated by a slash. Parameters specific for individual applications are described in sections dedicated to individual applications. We have already mentioned that the criterion for the computational complexity is the number of epochs from Definition \ref{def:epoch}. The standard criterion for the algorithm progress is the number of updates of the decision variable. Since during one epoch we perform $\frac{S}{\smin}$ updates, this is equal to
\begin{equation}\label{eq:defin_nupdate}
\nupdate = S\frac{\nepoch}{\smin}.
\end{equation}
In other words, if $S$ is fixed, the progress of Algorithm \ref{alg:lambda} should be constant (for sufficiently large minibatches) in $\frac{\nepoch}{\smin}$.

\begin{table}[!ht]
\centering
\caption{Summary of the used parameters: allowed failure level $\eps$, size of the minibatch $\smin$, number of epochs $\nepoch$, the stepsize $\alpha$ and the initial penalization parameter $\lambda^1$.}
\label{tab:data2}
\begin{tabular}{@{}llllll@{}}
\toprule
& $\eps$ & $\smin$ & $\nepoch$ & $\alpha$ & $\lambda^1$\\
Fishing & $0.2$ & $1000/100$ & $10/1$ & $10^{-2}$ & $10^1$ \\
Separator & $0.1$ & $1000/100/10$ & $20$ & $10^{-4}$ & $10^3$ \\
Gas network & $0.15$ & $2000/1000/500/250/100$ & $40/20/10/5/2$ & $10^{-4}$ & $10^{-3}$ \\
\bottomrule
\end{tabular}
\end{table}

All codes were implemented in Matlab. The only exception was the merging Algorithm \ref{alg:sort} where we spent hours of cursing and thousands of tears before we finally managed to implement it in C.

\subsection{Application 1: Optimal control of fish population}\label{sec:app2_1}

For the fishing application, the prices $p,d,c$ follow a time-dependent multivariate random walk. More precise generation process and other random variables are described in Appendix \ref{app:app1}. The time interval was chosen as $[0,10]$, which we discretized into $n=100$ and $n=1000$ time steps. Note that this number also equals to the number of decision variables. We repeated the experiment with $(\smin,\nepoch)=(100,1)$ and $(\smin,\nepoch)=(1000,10)$. Other parameters were chosen as described in Table \ref{tab:data2}.

We compare the total time for all values of the penalization parameter $\lambda$ in Table \ref{tab:app1_res1}. First, the total time is relatively small, for example for a problem with $n=1000$ decision variables and $S=100000$ scenarios, we need only $106$ seconds to solve the problem. Note that evaluations of the objective $f$ and the constraints $g$, which corresponds to solving the ODE, is the most time-consuming computation while the time to find the quantile is small.

\begin{table}[!ht]
\caption{Summary of the needed time for the fishing application from Section \ref{sec:app1_1}: Total time, time to evaluate the objective $f$ and the constraints $g$ and the time to find the quantile.}
\label{tab:app1_res1}
\centering
\begin{tabular}{@{}lllllllll@{}}
\toprule
\multicolumn{5}{c}{} & \multicolumn{4}{c}{Time [s]} \\\cmidrule{6-9}
$n$ & $S$ & $\smin$ & $\nepoch$ & $n_{\rm updates}$ & Total & Eval $f$ & Eval $g$ & Quantile \\
\midrule
$100$ & $100000$ & $100$ & $1$ & $1000$ & $17.1$ & $4.4$ & $2.0$ & $7.9$ \\
$100$ & $100000$ & $1000$ & $10$ & $1000$ & $100.2$ & $49.5$ & $24.0$ & $8.0$ \\
$1000$ & $100000$ & $100$ & $1$ & $1000$ & $106.1$ & $53.5$ & $7.2$ & $7.0$ \\
$1000$ & $100000$ & $1000$ & $10$ & $1000$ & $2321.6$ & $1439.4$ & $657.9$ & $14.1$ \\
\bottomrule
\end{tabular}
\end{table}

According to \eqref{eq:defin_nupdate}, the parameter choice resulted in constant number of decision variable updates $\nupdate=1000$. Thus, we expect the solutions to look similar. This is confirmed in Figure \ref{fig:fishing2}. Its left part depicts the fishing rate $u$ for all four parameter settings while the right part depicts the number of fish $x$ for randomly selected scenarios. These solutions look similar and can be simply interpreted: In interval $[0,2]$ the process is driven to a steady state which is kept with relatively constant fishing and the amount of fish in the interval $[2,7]$. Since the steady state has a smaller amount of fish than the desired end-time level $\xdes=1.5$, the fishing rate drops in the interval $[7,10]$ and the number of fish increases.

Looking at Figure \ref{fig:fishing2} again, we would like to emphasize the biggest advantage of the stochastic descent: Even though for each $\lambda$ we performed only $\nepoch=1$ evaluation of $g(\cdot,\xi)$ on the whole dataset, we performed $\nupdate=1000$ updates of the decision variable. This resulted in a high-quality solution in only $106$ seconds.

\begin{figure}[!ht]
\begin{tikzpicture}
  \pgfplotsset{small,samples=10}
  \begin{groupplot}[group style = {group size = 2 by 1}, grid=major, grid style={dashed, gray!50}]
      \nextgroupplot[xlabel={Time}, xmin=0, xmax=10, title = {}, ylabel={Fishing rate}, legend style = legendStyleA]
          \addplot [solid] table[x index=0, y index=1] {\tabAA}; \addlegendentry{$n=100$, $\smin=100$}
          \addplot [dashed] table[x index=0, y index=1] {\tabAB}; \addlegendentry{$n=100$, $\smin=1000$}
          \addplot [dotted] table[x index=0, y index=1] {\tabAC}; \addlegendentry{$n=1000$, $\smin=100$}
          \addplot [dashdotted] table[x index=0, y index=1] {\tabAD}; \addlegendentry{$n=1000$, $\smin=1000$}%
      \nextgroupplot[xlabel={Time}, xmin=0, xmax=10, title = {}, ylabel={Number of fish}]
          \foreach \tikzN in {2,...,51} {
          \addplot [solid,gray!50] table[x index=0, y index=\tikzN] {\tabAE};}%
          \addplot [solid,black,thick] table[x index=0, y index=1] {\tabAE}; %
      \end{groupplot}
\node at ($(group c1r1) + (-0.55,-1.1)$) {\ref{grouplegendA}};       
\end{tikzpicture}
\caption{The fishing rate $u$ (left) and the number of fish $x$ (right) for the fishing application from Section \ref{sec:app1_1}. Note that the desired fish level (bold line) is prescribed only at the final time.}
\label{fig:fishing2}
\end{figure}
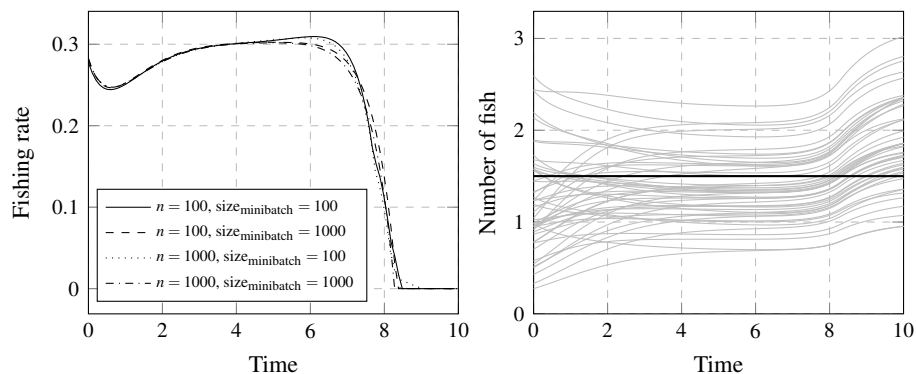

\subsection{Application 2: Optimal control of electrostatic separator}

We recall that the electrostatic separator takes two types of plastic or metal particles, triboelectrostatically charges them and exposes them to the electrostatic field. Due to the gravity, the particles fall downwards while due to the opposite polarities, they separate. At the bottom of the separator, there are three bins, one for each type and one for middling which is reseparated. The goal is to minimize the voltage such that the separation accuracy is at least $1-\eps$. The decision variables are the voltage $U$, the position of both electrodes $d_l$, $d_r$. The random variables are the particle mass $m$, charge $Q$ and its initial position $s_x(0),s_y(0)$. More detailed information is presented in Appendix \ref{app:app2}.

We will show the dependence of results on the number of scenarios $S$. In Figure \ref{fig:separator1}, we fix a design and repeatedly compute the separation accuracy on a randomly selected minibatch with $\smin\in\{10,100,1000\}$ and construct histograms. The variance is rather large, especially for the two smaller values. This plays a crucial role in our algorithm since we update the constraint only on a minibatch and smaller minibatches bring a rather large error to the computation. 

\begin{figure}[!ht]
\begin{tikzpicture}
  \pgfplotsset{footnotesize,samples=10}
  \begin{groupplot}[group style = {group size = 3 by 1, horizontal sep = 12pt}, grid=major, grid style={dashed, gray!50}, xmin=0.5, xmax=1, ymin=0, ymax=400, ybar, xtick pos=left]
      \nextgroupplot[ylabel={}]
          \addplot +[black, fill=gray, hist={bins=11, data min=0.5, data max = 1}] table [y index=0] {\tabB};
      \nextgroupplot[ylabel={}, yticklabels={}]
          \addplot +[black, fill=gray, hist={bins=11, data min=0.77, data max = 0.97}] table [y index=1] {\tabB};
      \nextgroupplot[ylabel={}, yticklabels={}]
          \addplot +[black, fill=gray, hist={bins=11, data min=0.868, data max = 0.928}] table [y index=2] {\tabB};
      \end{groupplot}    
\end{tikzpicture}
\caption{The histograms of the separation accuracy for minibatch sizes $\smin=10$ (left), $\smin=100$ (middle) and $\smin=1000$ (right). These minibatch sizes were used in the computations.}
\label{fig:separator1}
\end{figure}
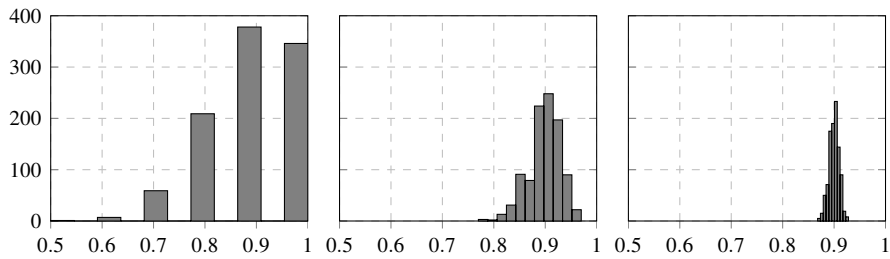

In Table \ref{tab:separator1} we depict the results for different values of the number of scenarios $S$ and the size of the minibatch $\smin$. Since we kept the number of epochs $\nepoch=20$ constant, this due to \eqref{eq:defin_nupdate} resulted in constant number of updates $\nupdate=2000$ of the decision variable. We show the optimal voltage $U$, the separator width $d_r-d_l$, the obtained failure levels $\eps_{\rm data}$ and $\eps_{\rm true}$ on the training data and testing data, respectively. Here, by training data, we understand the data on which the algorithm was trained while the testing data refer to computing the failure level on a large number of randomly generated scenarios which the algorithm has not used before. The last two columns refer to the total time and the time needed to evaluate the separation accuracy.

We observe several things. First, the Coulomb force $F_c$ is proportional to $\frac{U}{d_r-d_l}$. Since this force is the main force in the horizontal direction and the particles are supposed to fall into the correct bins, this ratio is constant. Second, the computed failure level $\eps_{\rm data}$ equals to the requested $\eps=0.1$ while the true failure level $\eps_{\rm true}$ is larger when the number of scenarios $S$ used for training is smaller. This makes sense as the algorithm overfitted to the training data.

\begin{table}[!ht]
\caption{The results for the separator application from Section \ref{sec:app1_2}. Based on the scenario size, it depicts the optimal voltage $U$, the optimal separator width $d_r-d_l$, the ratio $\frac{U}{d_r-d_l}$ proportional to the Coulomb force $F_c$, the failure level on the training data $\eps_{\rm data}$, the failure level outside of the training data $\eps_{\rm true}$, the total time and the time needed to evaluate the particle movement inside the separator.}
\label{tab:separator1}
\centering
\begin{tabular}{@{}lllllllll@{}}
\toprule
$S$ & $\smin$ & $U$ & $d_r-d_l$ & $U / (d_r-d_l)$ & $\eps_{\rm data}$ & $\eps_{\rm true}$ & Time [s] & Time $g$ [s] \\
\midrule
$1000$ & $10$ & $19308$ & $0.178$ & $1.088$ & $0.101$ & $0.116$ & $36$ & $23$ \\
$10000$ & $100$ & $22337$ & $0.211$ & $1.058$ & $0.100$ & $0.101$ & $69$ & $51$ \\
$100000$ & $1000$ & $23732$ & $0.227$ & $1.043$ & $0.100$ & $0.100$ & $531$ & $480$ \\
\bottomrule
\end{tabular}
\end{table}

In Figure \ref{fig:separator2} we show the obtained separator. Since the black particles are supposed to fly left and the grey particles are supposed to fly right, the algorithm has produced a good solution. The good results as presented in Table \ref{tab:separator1} and Figure \ref{fig:separator2} together with the large variance from Figure \ref{fig:separator1} indicates that it is indeed necessary to consider the delayed values as in \eqref{eq:update_z} to properly update the quantile.

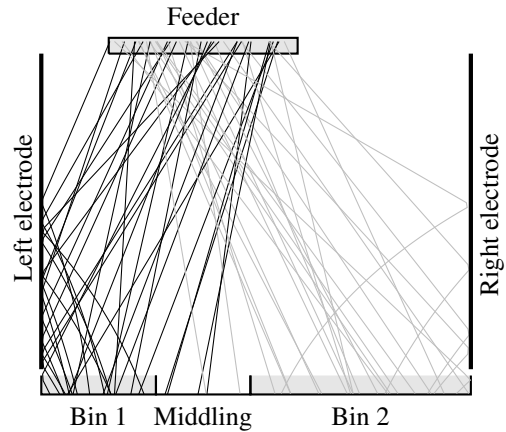
\begin{figure}[!ht]
\centering
\begin{tikzpicture}
\pgfplotsset{samples=10}
\begin{axis}[axis lines=none]
\addplot [black, thick, no marks, fill=gray!20] coordinates {(-0.05,0) (0.05,0) (0.05,0.05) (-0.05,0.05) (-0.05,0)};
\addplot [black, thick, no marks, fill=gray!20] coordinates {(-0.0857,-1.02) (-0.0857,-1.08)(-0.025,-1.08) (-0.025,-1.02)};
\addplot [black, thick, no marks, fill=gray!20] coordinates {(0.025,-1.02) (0.025,-1.08)(0.1417,-1.08) (0.1417,-1.02)};
\addplot [black, thick, no marks] coordinates {(-0.025,-1.02) (-0.025,-1.08)(0.025,-1.08) (0.025,-1.02)};
\foreach \tikzN in {0,...,24} {
\addplot [solid,gray!50] table[x index=\tikzN, y index=25+\tikzN] {\tabDA};
\addplot [solid,black] table[x index=\tikzN, y index=25+\tikzN] {\tabDB};}           
\addplot [black, ultra thick, no marks] coordinates {(-0.0857,0) (-0.0857,-1)};
\addplot [black, ultra thick, no marks] coordinates {(0.1417,0) (0.1417,-1)};
\node [above=1] at (axis cs: 0,0.05) {Feeder};
\node [above=1,rotate=90] at (axis cs: -0.0857,-0.5) {Left electrode};
\node [below=1,rotate=90] at (axis cs: 0.1417,-0.5) {Right electrode};
\node [below=1] at (axis cs: 0,-1.08) {Middling};
\node [below=1] at (axis cs: -0.0553,-1.08) {Bin 1};
\node [below=1] at (axis cs: 0.0833,-1.08) {Bin 2};
\end{axis}
\end{tikzpicture}
\caption{The schema of the free-fall electrostatic separator. Two types of plastics are charged with opposite polarities and placed into the feeder. Due to the gravity, they fall downwards and due to the electrostatic field between the electrodes, they separate.}
\label{fig:separator2}
\end{figure}

\subsection{Application 3: Optimal design of gas network}

For the gas network application we took the same setting as in \cite{adam.branda.heitsch.henrion.2018}. The schema of the used network with $1$ entry and $11$ exit nodes is depicted in Figure \ref{fig:gas}. The results are depicted in Table \ref{tab:app3_res1} (objective $f$) and Table \ref{tab:app3_res2} (failure level). As expected due to \eqref{eq:defin_nupdate}, the numbers on all diagonals offer comparable results.

\begin{figure}[!ht]
  \centering
  \includegraphics[width=0.5\textwidth]{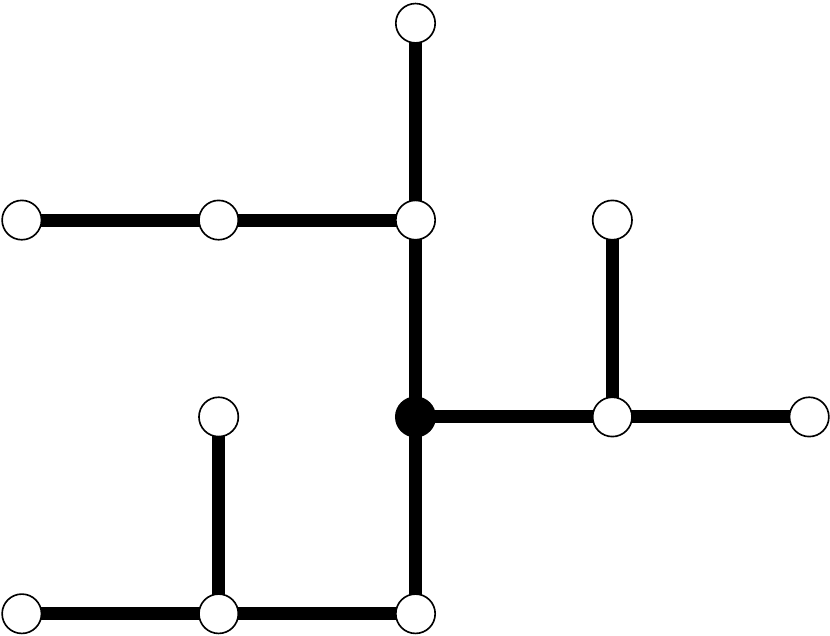}
  \caption{Schema for the gas network from Section \ref{sec:app1_3}.}
  \label{fig:gas}
\end{figure}

We would like to compare our results with the one in \cite{adam.branda.heitsch.henrion.2018}. While we managed to obtain the optimal objective $3144.68$ with failure level $0.1502$, the authors in \cite{adam.branda.heitsch.henrion.2018} reported the objective $3145.75$ with failure level $0.1500$. We can conclude that our method provides a comparable if not a better solution than the one reported earlier. Moreover, the computational time was less than one minute which is significantly faster than the time reported in \cite{adam.branda.heitsch.henrion.2018}.



\begin{table}[!ht]
\caption{The obtained objective $f$ for the gas network application from Section \ref{sec:app1_3}.}
\label{tab:app3_res1}
\centering
\begin{tabular}{@{}lrrrrr@{}}
\toprule
 & $100$ & $250$ & $500$ & $1000$ & $2000$ \\
\midrule
$2$ & $3248.80$ & $3169.86$ & $246478.36$ & $810451.46$ & $1202081.95$ \\
$5$ & $3145.20$ & $3145.04$ & $3221.21$ & $107410.95$ & $632311.21$ \\
$10$ & $3144.68$ & $3145.17$ & $3145.20$ & $3432.14$ & $110977.10$ \\
$20$ & $3144.95$ & $3145.56$ & $3145.80$ & $3145.10$ & $3584.16$ \\
$40$ & $3144.65$ & $3144.97$ & $3144.62$ & $3144.85$ & $3144.31$ \\
\bottomrule
\end{tabular}
\end{table}

\begin{table}[!ht]
\caption{The obtained failure level $\eps_{\rm data}$ for the gas network application from Section \ref{sec:app1_3}.}
\label{tab:app3_res2}
\centering
\begin{tabular}{@{}llllll@{}}
\toprule
 & $100$ & $250$ & $500$ & $1000$ & $2000$ \\
\midrule
$2$ & $0.113$ & $0.140$ & $0.412$ & $0.829$ & $0.965$ \\
$5$ & $0.150$ & $0.150$ & $0.153$ & $0.317$ & $0.756$ \\
$10$ & $0.150$ & $0.150$ & $0.150$ & $0.157$ & $0.333$ \\
$20$ & $0.150$ & $0.150$ & $0.150$ & $0.150$ & $0.159$ \\
$40$ & $0.150$ & $0.150$ & $0.150$ & $0.150$ & $0.150$ \\
\bottomrule
\end{tabular}
\end{table}

\paragraph{Acknowledgements} We would like to thank Holger Heitsch for providing us data for the gas network application.

This work was supported by National Natural Science Foundation of China (Grant No. 61850410534), Shenzhen Peacock Plan (Grant No. KQTD2016112514355531) and the Grant Agency of the Czech Republic (GA17-08182S, 19-28231X).

\bibliographystyle{abbrv}
\bibliography{Bibliography}

\appendix

\section{Properties of the quantile function $q(x)$}\label{app:q}

\begin{lemma}\label{lemma:q}
Let $g(\cdot,\xi_i)$ be Lipschitz continuous for all $i$. Then $q$ is Lispchitz continuous. Moreover, if $i(\hat x)$ is unique and if $g(\cdot,\xi_{i(\hat x)})$ is differentiable around some $\hat x$, then $q$ is differentiable at $\hat x$.
\end{lemma}
\begin{proof}
Fix any $\bar x$. Due to \eqref{eq:defin_q}, there is a disjoint partition $I$, $J_1$ and $J_2$ of $\{1,\dots,S\}$ such that
$$
\aligned
g(\bar x,\xi_i) &= q(\bar x)\text{ for all }i\in I, \\
g(\bar x,\xi_i) &< q(\bar x)\text{ for all }i\in J_1, \\
g(\bar x,\xi_i) &> q(\bar x)\text{ for all }i\in J_2.
\endaligned
$$
Due to the assumed continuity of $g(\cdot,\xi_i)$, there is a neighborhood of $\bar x$ such that for all $y$ from this neighborhood, there exist an index $i(y)\in I$ such that
\begin{equation}\label{eq:q_equality_2}
g(y,\xi_{i(y)}) = q(y).
\end{equation}
But since $i(y)\in I$, we have
$$
\nrm{q(y) - q(\bar x)} = \nrm{g(y,\xi_{i(y)}) - g(\bar x,\xi_{i(y)})} \le L_g\norm{y-\bar x},
$$
where $L_g$ is the Lispchitz constant of $g(\cdot,\xi_{i(y)})$. Thus, $q$ is Lipschitz continuous. If the index in \eqref{eq:q_equality_1}, then the index in \eqref{eq:q_equality_2} is unique and $q$. The second part follows from the fact that $I$ is unique whenver $i(\hat x)$ is unique.
\end{proof}

\added{
\begin{proof}[Proof of Lemma \ref{lemma:q_conv}]\ 
Denote the Lipschitz constants of $f$ and $g$ by $L_f$ and $L_g$, respectively. The notation $\nabla_x f(x,\xi_i)$ means either the gradient and if $f(\cdot,\xi_i)$ is not differentiable at $x$, then any element of the Clarke subdifferential. Note that in both cases we have $\norm{\nabla_x f(x,\xi_i)}\le L_f$. We will use the same notation for $\nabla_x g(x,\xi_i)$.

Since $x^k$ is uniformly bounded due to compactness of $X$ and since $g(\cdot,\xi_i)$ is continuous, there is some $B_g$ such that $\nrm{g(x^k,\xi_i)}\le B_g$ for all $i$ and $k$. Since $X$ is convex and since the projection onto a convex set is $1$-Lipschitz, we observe that
$$
\aligned
\norm{x^{k+1} - x^k} &= \norm{P_X(y^{k+1}) - P_X(x^k)}\le \norm{y^{k+1} - x^k} \\
&= \alpha^k\left\| \frac1{\nrm{I^k}} \sum_{i\in I^k}\nabla_x f(x^k,\xi_i) + \lambda \phi'(q^k)\nabla_x g(x^k,\xi_{i^k}) \right\| \\
&\le \alpha^k(L_f + \lambda \max\{q^k,0\}L_g) \le \alpha^k(L_f + \lambda B_gL_g)
\endaligned
$$

We recall that $q(x^k)$ is the quantile of $\{g(x^k,\xi_i)\}_{i=1}^S$ while $q^k$ is the quantile of $\{z_i^k\}_{i=1}^S$. Moreover, from \eqref{eq:update_z} and the way in which we choose minibatches we observe that for all $i$ there exists some $j\in\{0,2\nmin-1\}$ such that $z_i^k=g(x^{k-j},\xi_i)$. This implies
$$
\aligned
\nrm{g(x^k,\xi_i)-z_i^k} &= \nrm{g(x^k,\xi_i)-g(x^{k-j},\xi_i)} \le L_g\norm{x^k-x^{k-j}}\le L_g\sum_{l=0}^{j-1}\norm{x^{k-l}-x^{k-l-1}} \\
&\le L_g\sum_{l=0}^{2\nmin}\norm{x^{k-l}-x^{k-l-1}}\le L_g(L_f + \lambda B_gL_g)\sum_{l=0}^{2\nmin}\alpha^{k-l-1} \\
\endaligned
$$
This implies that $\nrm{g(x^k,\xi_i)-z_i^k}\to 0$ whenever $\alpha^k\to 0$ or whenever $\{x^k\}$ is convergent. But due to the definition of the quantile this means $\nrm{q(x^k)-q^k}\to 0$.
\end{proof}
}

\begin{remark}\label{remark:indices}
\added{
In this remark, we provide a very informal analysis showing that if there are multiple indices $i\in I$ such that $g(x,\xi_i)=q(x)$ for all $i\in I$, then $q$ may still be differentiable at $x$ and its gradient equals to a convex combination of $\nabla_x g(x,\xi_i)$ for $i\in I$.

Consider the case when $\xi\in\R$ is an absolutely continuous random variable and when $g$ is differentiable at $(x,\xi_i)$ for all $i\in I$. Since we aim to compute $\nabla q(x)$, we perturb $x$ by $\delta x$ and $t=q(x)$ by some $\delta t$. Then we locally around $\xi_i$ approximate 
$$
\aligned
\{\delta\xi|\ g(x,\xi_i+\delta\xi)\le t\} &\approx \{\delta\xi|\ g(x,\xi_i)+ \nabla_\xi g(x,\xi_i)\delta\xi \le t\} \\
&= \{\delta\xi|\ \nabla_\xi g(x,\xi_i)\delta\xi \le t - g(x,\xi_i)\}, \\
\{\delta\xi|\ g(x+\delta x,\xi_i+\delta\xi)\le t+\delta t\} &\approx \{\delta\xi|\ g(x,\xi_i)+ \nabla_x g(x,\xi_i) \delta x + \nabla_\xi g(x,\xi_i)\delta\xi \le t +\delta t\} \\
&= \{\delta\xi|\ \nabla_\xi g(x,\xi_i)\delta\xi \le t - g(x,\xi_i)- \nabla_x g(x,\xi_i) \delta x +\delta t\}
\endaligned
$$
For simplicity assume that $\nabla_\xi g(x,\xi_i)>0$ for all $i\in I$. Note that the whole analysis can be performed for $\nabla_\xi g(x,\xi_i)\neq 0$. Then we locally around $\xi_i$ have
\begin{equation}\label{eq:rem2}
\aligned
\{\delta\xi|\ g(x,\xi_i+\delta\xi)\le t\} &\approx \{\delta\xi|\ \delta\xi \le \nabla_\xi^{-1}g(x,\xi_i)\left(t - g(x,\xi_i)\right)\}, \\
\{\delta\xi|\ g(x+\delta x,\xi_i+\delta\xi)\le t+\delta t\} &\approx \{\delta\xi|\ \delta\xi \le \nabla_\xi^{-1}g(x,\xi_i)\left(t - g(x,\xi_i)- \nabla_x g(x,\xi_i) \delta x +\delta t\right)\}.
\endaligned
\end{equation}

Since $\xi$ is absolutely continuous, we have
$$
\PP(g(x+\delta x,\xi)\le t+\delta t) = \PP(g(x,\xi)\le t) = 1-\eps,
$$
which implies
\begin{equation}\label{eq:rem3}
\PP(g(x+\delta x,\xi)\le t+\delta t) - \PP(g(x,\xi)\le t) = 0.
\end{equation}
The local change of the quantity on the left-hand side of \eqref{eq:rem3} around $\xi_i$ with $i\in I$ amounts due to \eqref{eq:rem2} and the definition of the distribution function to
\begin{equation}\label{eq:rem4}
\aligned
&\PP(g(x+\delta x,\xi_i)\le t+\delta t) - \PP(g(x,\xi_i)\le t) \\
&\qquad= F_\xi\left(\xi_i+\nabla_\xi^{-1}g(x,\xi_i)\left(t - g(x,\xi_i) -  \nabla_x g(x,\xi_i) \delta x +\delta t\right)\right) - F_\xi\left(\xi_i+\nabla_\xi^{-1}g(x,\xi_i)\left(t - g(x,\xi_i)\right)\right).
\endaligned
\end{equation}
For small $\delta x$, the quantity on the left-hand side of \eqref{eq:rem3} changes only around $\xi_i$ for $i\in I$. Thus, due to \eqref{eq:rem3} and \eqref{eq:rem4} we have
$$
\sum_{i\in I} F_\xi\left(\xi_i+\nabla_\xi^{-1}g(x,\xi_i)\left(t - g(x,\xi_i) -  \nabla_x g(x,\xi_i) \delta x +\delta t\right)\right) - \sum_{i\in I} F_\xi\left(\xi_i+\nabla_\xi^{-1}g(x,\xi_i)\left(t - g(x,\xi_i)\right)\right)=0,
$$
where $F_\xi$ is the distribution function of $\xi$. Assuming that $F_\xi$ is differentiable, the first-order approximation with respect to $\xi$ of the previous equality reads
$$
\sum_{i\in I}\nabla_\xi F_\xi\left(\xi_i+\nabla_\xi^{-1}g(x,\xi_i)\left(t - g(x,\xi_i)\right)\right)\left(\nabla_\xi^{-1}g(x,\xi_i)\left(\delta t - \nabla_x g(x,\xi_i) \delta x\right)\right)=0,
$$
from which we deduce
$$
\delta t = \sum_{i\in I}\frac{a_i}{\sum_{j\in I}a_j}\nabla_x g(x,\xi_i) \delta x
$$
with
$$
a_i = \nabla_\xi F_\xi\left(\xi_i+\nabla_\xi^{-1}g(x,\xi_i)\left(t - g(x,\xi_i)\right)\right)\nabla_\xi^{-1}g(x,\xi_i).
$$

The analysis suggests that
$$
\nabla q(x) = \sum_{i\in I}\frac{a_i}{\sum_{j\in I}a_j}\nabla_x g(x,\xi_i).
$$
Since $a_i\ge 0$ as $F_\xi$ is the distribution function and thus non-decreasing, we obtain that $\nabla q(x)$ is a convex combination of $\nabla_x g(x,\xi_i)$ for $i\in I$.
}
\end{remark}

\section{Efficient algorithm for computing the quantile}\label{app:quantile}

\added{
In Algorithm \ref{alg:sgd} we used an on-the-shelf algorithm to compute the quantile $q^k$. This can be done in $O(S)$. However, there is a more efficient method which we present here. Its complexity is also $O(S)$ but it needs only one pass through the array.} Its basic idea is to keep $z^k$ sorted at every iteration. Due to \eqref{eq:update_z}, $z^{k+1}$ differs from $z^k$ only on the minibatch $I^{k+1}$. Since $z^k$ is already sorted, it suffices to sort the new values $g(x^{k+1},\xi_i)$ on $I^{k+1}$ and then merge these two sorted arrays. This can be done in one pass through both arrays. Then we obtain the sorted version of $z^{k+1}$ and the quantile equals to index $\lceil S(1-\eps)\rceil$ of this sorted array.

Now we formalize this idea. At iteration $k$ we know the sorting permutation $\pi$ of $\{1,\dots,S\}$ which sorts $z^k$ into $s^k$ and we compute the sorting permutation $\varphi$ of $I^{k+1}$ which sorts $g_i^k:=g(x^{k+1},\xi_i)$, $i\in I^{k+1}$ into $h_j^k$. Namely, we have
\begin{equation}\label{eq:sort}
\aligned
s_i^k &= z^k_{\pi(i)}, &&\pi(i)\in \{1,\dots,S\}, \\
h_j &= g_{\varphi(j)}, &&\varphi(j)\in I^{k+1}.
\endaligned
\end{equation}
According to \eqref{eq:update_z}, we want to update $z^k$ only on minibatch $I^{k+1}$. Then we have
\begin{equation}\label{eq:sort1}
z_i^{k+1} = \begin{cases} g_i^k=g(x^{k+1},\xi_i) &\text{if }i\in I^{k+1}, \\ z_i^k &\text{otherwise.} \end{cases}
\end{equation}
Thus, we have two sorted arrays $s^k$ and $h^k$ from \eqref{eq:sort} and we need to merge them together.

To this aim, we will employ three indices: $i$ will run on $s^k$, $j$ will run on $h^k$ and finally $l$ will run on $s^{k+1}$. We initialize all indices to $i=j=l=1$. In every iteration, we first check whether $\pi(i)\in I^{k+1}$. If this is the case, $z_{\pi(i)}^k$ was replaced by $g_{\pi(i)}$ in \eqref{eq:sort1}. Thus, we are not interested in $s_i^k=z_{\pi(i)}^k$, we increase $i$ by one and repeat. In the opposite case of  $\pi(i)\notin I^{k+1}$ we set
$$
s_l^{k+1} = \begin{cases} s_i^k &\text{if }s_i^k \le h_j^k, \\ h_j^k &\text{otherwise.} \end{cases}
$$
In other words, we insert to $s_l^{k+1}$ the minimum of $s_i^k$ and $h_j^k$. Since both $s^k$ and $h$ are sorted, $s^{k+1}$ is sorted as well. Now, we need to obtain a permutation $\psi$ mapping $z^{k+1}$ into its sorted variant $s^{k+1}$, namely it needs to satisfy
\begin{equation}\label{eq:sort3}
s_l^{k+1}=z_{\psi(l)}^{k+1}.
\end{equation}
Due to \eqref{eq:sort}, this can be simply obtained by
$$
\psi(l) = \begin{cases} \pi(i) &\text{if }s_i^k \le h_j^k, \\ \varphi(j) &\text{otherwise.} \end{cases}
$$
Since $s_l^{k+1}$ has been filled, we increase $l$ by one. Finally, if $s_l^{k+1}$ was filled by $s_i^k$, we increase $i$ by one and if $s_l^{k+1}$ was filled by $h_j^k$, we increase $j$ by one. In both cases, we repeat the whole precedure. After sorting $z^{k+1}$, the quantile can be computed as
$$
q^{k+1}=s_{\rm index}^{k+1},
$$
where $\rm{index} = \lceil S(1-\eps)\rceil$.

To summarize, the procedure above takes as input the sorted array $s^k$ from the previous iteration and the permutation $\pi$ satisfying \eqref{eq:sort}. It then replaces values $g$ on indices $I^{k+1}$ and returns the sorted array $s^{k+1}$ and the permutation satisfying \eqref{eq:sort3}. These outputs can be directly used as inputs for the next iteration and thus, the sorting can be done efficiently. Note that it is not even necessary to store the unsorted array $z^k$.

This whole procedure is written down in Algorithm \ref{alg:sort}. We made two small changes to the procedure described above. First, in steps \ref{alg:sort_l1}-\ref{alg:sort_l2} we needed to take care about indices overflowing the arrays. Second, it may be time-consuming to check whether $\pi(i)\in I^{k+1}$. For this reason, in step \ref{alg:sort_l0} we assume that $I^{k+1}=\{\iLow,\dots,\iHigh\}$ and then insert the randomness into selection minibatches by permuting $\xi$. \added{This new computation of the quantile can be directly inserted in Algorithm \ref{alg:sgd}. Note that the only change is a faster computation of the quantile but the computed quantile is the same.}

\begin{algorithmaux}[!ht]
\begin{algorithmic}[1]
\Require sorted arrays $s^k$ and $h^k$ with the corresponding permutations \eqref{eq:sort}
\Ensure sorted array $s^{k+1}$ with the corresponding permutation \eqref{eq:sort3}
\State $i = j = k \gets 1$
\While {\textbf{true}}
\If {$\pi(i) \ge \iLow$\textbf{ and } $\pi(i) \le \iHigh$} \label{alg:sort_l0} \Comment{Value was replaced, ignore it}
\State $i\gets i+1$
\State \textbf{if }$i> S$\textbf{ then break while}
\Else
\If {$s_i^k \le h_j$} \Comment{Add $s_i^k$ to the new array}
\State $s_l^{k+1}\gets s_i^k$
\State $\psi(l)\gets \pi(i)$
\State $l\gets l+1$, $i\gets i+1$
\State \textbf{if }$i> S$\textbf{ then break while}
\Else \Comment{Add $h_j$ to the new array}
\State $s_l^{k+1}\gets h_j^k$
\State $\psi(l)\gets \varphi(j)$
\State $l\gets l+1$, $j\gets j+1$
\State \textbf{if }$j> \iHigh-\iLow+1$\textbf{ then break while}
\EndIf
\EndIf
\EndWhile
\If {$i<S$} \label{alg:sort_l1} \Comment{Handle the remaining terms}
\State Array $h^k$ has been inserted to $s^{k+1}$. Insert the remaining of $s^k$ to $s^{k+1}$ in a similar way as above.
\Else
\State Array $s^k$ has been inserted to $s^{k+1}$. Insert the remaining of $h^k$ to $s^{k+1}$ in a similar way as above.
\EndIf \label{alg:sort_l2}
\State $q^{k+1}\gets s_{\rm index}^{k+1}$ 
\end{algorithmic}
\caption{For finding the quantile $q^{k+1}$ from \eqref{eq:update_q}}
\label{alg:sort}
\end{algorithmaux}

\section{Backpropagation}\label{app:backprop}

In this short section, we describe how the well-known method backpropagation \cite{hecht.1992} computes derivatives. Consider a general function
$$
F(u):=\sum_{t=0}^T f_t(u_t,x_t),
$$
where $x_0$ is given and we have
$$
x_{t+1} = h_t(u_t, x_t).
$$
Since for every $u$, we can uniquely compute $x$, we want to compute the derivative of the objective with respect to $u$. We have
$$
\frac{\partial F}{\partial u_t} = \nabla_u f_t(u_t,x_t) + \sum_{s=t+1}^T \nabla_x f_s(u_s,x_s)\frac{\partial x_s}{\partial u_t} = \nabla_u f_t(u_t,x_t) + \sum_{s=t+1}^T \nabla_x f_s(u_s,x_s)\frac{\partial x_s}{\partial x_{s-1}}\dots\frac{\partial x_{t+2}}{\partial x_{t+1}}\frac{\partial x_{t+1}}{\partial u_t}
$$
Based on this expression, define
$$
a_t := \sum_{s=t+1}^T \nabla_x f_s(u_s,x_s)\frac{\partial x_s}{\partial x_{s-1}}\dots\frac{\partial x_{t+2}}{\partial x_{t+1}}.
$$

Then with $a_T=0$, we have the chaining relation
$$
a_t = a_{t+1}\frac{\partial x_{t+2}}{\partial x_{t+1}} + \nabla_x f_{t+1}(x_{t+1},u_{t+1}) = a_{t+1}\nabla_x h_{t+1}(u_{t+1},x_{t+1}) + \nabla_x f_{t+1}(x_{t+1},u_{t+1})
$$
and the derivative can be computed as 
$$
\frac{\partial F}{\partial u_t} = \nabla_u f_t(u_t,x_t) + a_t\frac{\partial x_{t+1}}{\partial u_t} = \nabla_u f_t(u_t,x_t) + a_t\nabla_u h_t(x_t,u_t).
$$
Note that the function value $F(u)$ can be computed in one forward swipe and the Jacobian $\nabla J(u)$ can be computed in one backward swipe.

\section{Computation of Derivatives}\label{app:derivatives}

Similarly, the objective function and the constraints can be discretized into
\begin{equation}\label{eq:app1_functions}
\aligned
f(u) &:= f(u,x(u)) := \Delta t\sum_{t=0}^{T-1}\left( p_tu_tx_t - d_tu_t^2x_t^2 - c_tu_t \right), \\
g(u) &:= g(u,x(u)) := x_T.
\endaligned
\end{equation}

In this section, we compute the derivatives for functions \eqref{eq:app1_functions} form the fishing application. Since the derivatives of functions from the separator application can be computed in an identical way, we omit it here. Using the backpropagation method from Section \ref{app:backprop}, for $t=0,\dots,T-1$ we can compute the derivative by
$$
\aligned
\frac{\partial f}{\partial u_t} &= \Delta t\left( px_t - 2du_tx_t^2 - c \right)      - a_t \Delta t x_t,
\endaligned
$$
where $a_T=a_{T-1}=0$ and for $t=0,\dots,T-2$ we set
$$
a_t = a_{t+1}\left(1 + \Delta t\left(r - \frac 2Krx_{t+1} - u_{t+1}\right)\right) + \Delta t \left( pu_{t+1} - 2du_{t+1}^2x_{t+1} \right).
$$
Note that this is a discretized version of the adjoint equation which can be obtained by the techniques from optimal control theory \cite{clarke.2013}. Note that the computation above contains only vectors while a naive application of the chain rule would result in $\nabla_u f(u,x(u)) + \nabla_x f(u,x(u))\frac{\partial x}{\partial u}$, where $\frac{\partial x}{\partial u}$ is a matrix.
Similarly, for the terminal state function $g$ we have
$$
\frac{\partial g}{\partial u_t} = -b_t\Delta t x_t,
$$
where
$b_T=0$ and
$$
\aligned
b_{T-1} &= b_T\left(1 + \Delta t\left(r - \frac 2Krx_{T} - u_{T}\right)\right) + 1, \\
b_t &= b_{t+1}\left(1 + \Delta t\left(r - \frac 2Krx_{t+1} - u_{t+1}\right)\right).
\endaligned
$$

\section{Differences between the batch and stochastic gradient descents}\label{app:diff}

In Section \ref{sec:app1_1} we showed that in problem \eqref{eq:problem1} we can replace the random variable $p_t$ by its expectation $\EE p_t$. These two problems are equivalent when the batch gradient descent is applied. However, this is no longer true for the stochastic gradient descent. To show this, we proposed the following experiment. We fixed the optimized design and randomly generated one set with $\smin$ scenarios. Then we considered the same set where $p_t$ was replaced by its expectation $\EE p_t$. When we computed the difference in the values of the objective $f$, the differences were negligible. However, the difference in Jacobians of the objective $\nabla f$ was rather big. The value of the gradient is showed in Figure \ref{fig:fishing4} (left) while the relative difference between both settings is depicted in Figure \ref{fig:fishing4} (right). Due to this difference, the two problems above are no longer equivalent when the stochastic gradient descent is used.

\begin{figure}[!ht]
\begin{tikzpicture}
  \pgfplotsset{small,samples=10}
  \begin{groupplot}[group style = {group size = 2 by 1}, grid=major, grid style={dashed, gray!50}]
      \nextgroupplot[xlabel={Time}, xmin=0, xmax=10, title = {}, ylabel={Gradient}, legend style = legendStyleB]
          \addplot [solid] table[x index=0, y index=1] {\tabAF}; \addlegendentry{$\smin=100$}
          \addplot [dashed] table[x index=0, y index=2] {\tabAF}; \addlegendentry{$\smin=1000$}
      \nextgroupplot[xlabel={Time}, xmin=0, xmax=10, ymin=0, ymax=1, title = {}, ylabel={Gradient relative error}]
          \addplot [solid] table[x index=0, y index=3] {\tabAF};
          \addplot [dashed] table[x index=0, y index=4] {\tabAF}; 
      \end{groupplot}
\node at ($(group c1r1) + (-0.55,-1.1)$) {\ref{grouplegendB}};       
\node at ($(group c2r1) + (0.55,1.1)$) {\ref{grouplegendB}};       
\end{tikzpicture}
\caption{The gradient (left) and the relative error in gradient when the random prices are replaced by their expectation (right). Both figures are for a fixed design.}
\label{fig:fishing4}
\end{figure}
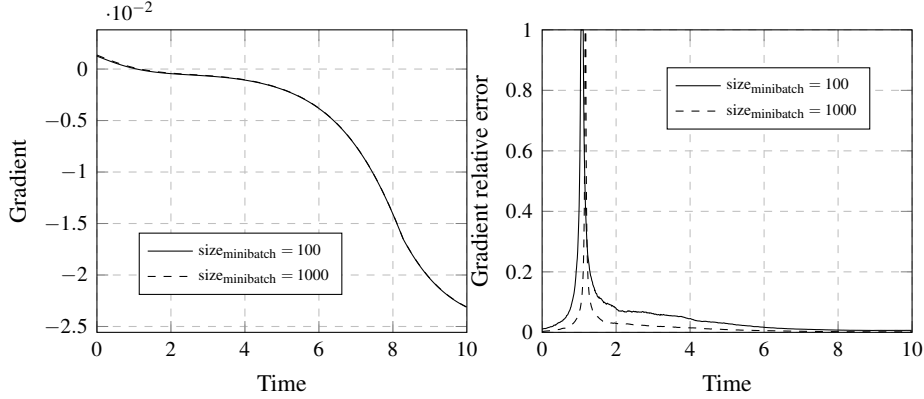

\section{Additional information for the numerical part}

In this section, we provide additional information about parameter choices from the numerical section.

\subsection{Application 1: Optimal control of fish population}\label{app:app1}

The random variables followed the $r\sim N(1,0.01)$, $K\sim N(2,0.25)$ and $x_0\sim U(\frac{K}4,K)$ distributions where $r$ was truncated at $0.01$ and $K$ was truncated at $1$. The random prices were initialized in $\bar p_t=2$, $\bar d_t=0.1$, $\bar c_t=1.5$ and then followed the process
$$
\bpm p_0 \\ d_0 \\ c_0\epm = \bpm 2 \\ 0.1 \\ 1.5\epm,\qquad \bpm p_{t+1} \\ d_{t+1} \\ c_{t+1}\epm = \bpm p_t \\ d_t \\ c_t \epm + N\left(\bpm 0\\ 0\\0\epm, \frac{0.01}{n}\bpm 1 & 0.025 &  0.5 \\ 0.025 &0.05 & 0.025\\ 0.5&  0.025 & 1\epm \right).
$$

\subsection{Application 2: Optimal control of electrostatic separator}\label{app:app2}

For the random variables, we loosely followed \cite{kacerovsky.2016}. For the positively changed particles, we had $m\sim N(6\cdot 10^{-6}\text{kg}, (2\cdot 10^{-6}\text{kg})^2)$ and $Q\sim N(5.5\cdot 10^{-11}\text{C}, (2.2\cdot 10^{-11}\text{C})^2)$ while for the negatively charged particles, we had $m\sim N(8\cdot 10^{-6}\text{kg}, (1.5\cdot 10^{-6}\text{kg})^2)$ and $Q\sim N(-5\cdot 10^{-11}\text{kg}, (2\cdot 10^{-11}\text{kg})^2)$. The mass was truncated at $1e-7$kg while there was no truncation for the charge. The initial position followed $s_x\sim U(-0.05\text{m},0.05\text{m})$ and $s_y\sim U(0\text{m},0.05\text{m})$. The separator was $1$m tall. The middle bin was located at $[-0.025\text{m},0.025\text{m}]$.

\end{document}